\documentclass[10pt]{my_NSP1}

\usepackage{url,floatflt}
\usepackage{helvet,times}
\usepackage{psfig,graphics}
\usepackage{mathptmx,amsmath,amssymb,bm}
\usepackage{float}
\usepackage[bf,hypcap]{caption}

\emergencystretch=\maxdimen
\def\no{\noindent}

\def\firstpage{1}
\def\thevol{?}
\def\thenumber{?}
\def\theyear{20??}

\newcommand{\ds}{\displaystyle}

\begin{document}

\titlefigurecaption{{\large \bf \rm Progress in Fractional Differentiation
and Applications}\\ {\it\small An International Journal}}

\title{Enlarged Controllability and Optimal Control
of Sub-Diffusion Processes with Caputo Fractional Derivatives}

\author{Touria Karite{$^{1}$}, Ali Boutoulout{$^1$} and Delfim F. M. Torres{$^{2,*}$}}

\institute{TSI Team, MACS Laboratory, 
Department of Mathematics \& Computer Science, 
Institute of Sciences, Moulay Ismail University, Meknes, Morocco\and
Center for Research and Development in Mathematics and Applications (CIDMA),
Department of Mathematics, University of Aveiro, 3810-193 Aveiro, Portugal}


\titlerunning{Enlarged Controllability and Optimal Control ...}

\authorrunning{T. Karite et al.}

\mail{delfim@ua.pt}


\received{4 Nov 2018}
\revised{16 Nov 2019}
\accepted{22 Nov 2019}


\abstracttext{We investigate the exact enlarged controllability and optimal control
of a fractional diffusion equation in Caputo sense. 
This is done through a new definition of enlarged 
controllability that allows us to extend available contributions. Moreover, 
the problem is studied using two approaches: a reverse Hilbert uniqueness 
method, generalizing the approach introduced by Lions in 1988, 
and a penalization method, which allow us to characterize the minimum energy control.}

\keywords{Fractional calculus and diffusion,
Caputo derivatives and enlarged controllability,
RHUM approach and minimum energy,
fractional optimal control,
zone and pointwise actuators.\\[2mm]
}

\maketitle


\section{\; Introduction}

The calculus of fractional order began more than three centuries ago. 
It was first mentioned by the celebrated Leibniz, in a letter replying
to l'H\^{o}pital, addressing the question whether the derivative
remains valid for a non-integer order. The subject has been developed 
by several mathematicians such as Euler, Fourier, Liouville, Grunwald, Letnikov 
and Riemann, among many others, up to the present day where many authors 
study such kind of operators and propose new fractional derivatives 
\cite{Book:IP:1999,Book:Kilbas:et:al:2006,Book:Mal:To:2012,%
Book:Malinowska:et:al:2015,Book:Baleanu:et:al:2017}. 
Over the last decades, fractional calculus has gained more and more attention 
because of its applications in various fields of science, such as physics, engineering, 
economics, and biology \cite{Paper:Ich:Nag:Koj:71,Paper:Machadoa:et:al:2006,%
Paper:Battaglia:et:al:2001,Paper:BAG:Tor:1983,Paper:Cat:Sor:2007,MyID:387}.

In control theory, several authors have been interested in fractional calculus 
since the sixties of last century. The first contributions generalize classical 
analytical methods and concepts for fractional order systems, such as the transfer 
function, frequency response, pole and zero analysis, and so on 
\cite{Book:Oustaloup:1983,Paper:Axt:Bis:1990}. Nowadays, fractional calculus 
is used in the field of automatic control to obtain more accurate models, 
to develop new control strategies and to improve the characteristics 
of control systems \cite{MyID:264,Orti11}.

In recent years, fractional order sub-diffusion systems have
been attracting the attention of many researchers because they have 
shown to have advantages over traditional integer order systems, 
as they can characterize more accurately anomalous diffusion processes 
in various real-world complex systems \cite{Paper:Met:Klaf:2000,%
Paper:Weiss:et:al:2004,Paper:Chen:et:al:2010,Paper:Luchko:2010}. 
In particular, fractional anomalous diffusion has been used to describe different physical 
scenarios, most prominently within crowded systems, for example protein diffusion 
within cells or diffusion through porous media. Time-fractional sub-diffusion 
has also been proposed as a measure of macromolecular crowding in the cytoplasm \cite{SP}.

It is worth to mention that the controllability of a fractional 
order sub-diffusion system can be reformulated as an infinite
dimensional control problem. Moreover, in case of diffusion systems, 
not all states can in general be reached
\cite{Book:Jai:Prit:1988,Paper:Ge:et:al:Grad:2016,Paper:Ge:et:al:2016}.
Because mathematical models of real systems are obtained from measures 
or from approximation techniques, which are very often affected by perturbations, 
if the solutions for such systems are only approximately known, then 
control problems subject to output fractional constraints are more realistic 
and more adapted for system analysis than the classical 
ones \cite{MyID:387,MyID:323}.

Many problems in modern science are commonly attacked with the help of 
optimization theory. Optimal control can be regarded as a branch of 
Mathematics whose goal is to improve the state variables of a control system
in order to maximize a benefit or minimize a given cost. This is applicable 
to practical situations, where state variables can be temperature, 
a velocity field, a measure of information, etc. This is the main reason 
why optimal control is an attractive research area for many scientists in 
various disciplines. Along the years, efficient optimization and optimal 
control methods have been developed in order to compute the solution 
of such fractional optimal control problems 
\cite{MyID:387,Paper:Tlacuahuac:Biegler:2014,Paper:Zhou:Casas:2014}.

Here we deal with the controllability problem of Caputo fractional 
diffusion equations when in presence of constraints on the state variables. 
This is related with the notion of enlarged controllability, which 
was first investigated by Lions in 1988 for hyperbolic systems 
\cite{Book:Lions:1988} and later developed for linear and semilinear parabolic systems
\cite{Paper:Zerrik:Ghafrani:2002,Paper:Karite:Boutoulout:2017,Paper:Karite:Boutoulout:2018,MyID:391}.
Moreover, we create a bridge with optimal control of systems described by 
fractional order differential equations. For that we prove enlarged 
controllability by means of a reverse HUM (Hilbert uniqueness method)
and make use of a penalization method, which allow us to characterize 
the minimum energy control. We consider the Caputo fractional derivative 
because it allows traditional initial and boundary conditions to be 
included in the formulation of the problem \cite{Nazare:Shahmorad}.
For related results with the Caputo--Fabrizio operators, we refer the reader
to the recent paper \cite{Karite2020}.

The remainder of the paper is organized as follows. 
Definitions and preliminaries on fractional calculus 
are given in Section~\ref{Sec:2}. In Section~\ref{Sec:3}, 
we characterize the exact enlarged controllability 
of the system. Our main results on the exact enlarged controllability 
are then proved in Section~\ref{Sec:4}, in two different cases: 
for zone and pointwise actuators. In Section~\ref{Sec:5}, an optimization 
problem for a system of fractional order is solved using a penalization method. 
We end up with Sections~\ref{Sec:examples} and \ref{Sec:6}, 
providing, respectively, some examples in both cases of actuators; and
conclusions of our work, pointing out some open questions 
deserving further investigations.


\section{\; Preliminaries}
\label{Sec:2}

Let $\Omega\subset\mathbb{R}^{n}$ be bounded with a smooth boundary $\partial{\Omega}$. 
For $T>0$, denote $Q=\Omega\times[0,T]$ and $\Sigma=\partial\Omega\times[0,T]$. 
We consider the following abstract fractional sub-diffusion 
system of order $\alpha\in(0,1)$:
\begin{eqnarray}
\label{sys1:eq1}
\begin{cases}
\ds {{}^{C}_{}D}^{\alpha}_{}y(t) = \mathcal{A}y(t)  + \mathcal{B}u(t), \quad & t\in [0,T],\\
y(0)=y_{_{0}}  \quad &\mbox{in}\quad D(\mathcal{A}),
\end{cases}
\end{eqnarray}
where ${}^{C}_{}D^{\alpha}_{}$ denotes the Caputo fractional order derivative 
(for details on Caputo fractional derivatives, see, e.g., \cite{Book:IP:1999,Book:AAK:HMS:JJT:2006}). 
The second order operator $\mathcal{A}$ is linear and with dense domain, 
such that the coefficients do not depend on $t$ and generates a $C_{0}$-semi-group 
$(S(t))_{_{t\geq 0}}$ on the Hilbert space $L^{2}(\Omega)$. We refer the reader 
to Engel and Nagel \cite{Book:KJE:RN:2006} and Renardy and Rogers \cite{Book:MR:RCR:2004} 
for properties on operator $\mathcal{A}$. In the sequel, we let
$\mathcal{D}(A)$ be the domain of the operator $\mathcal{A}$; 
$y\in L^{2}(0,T; L^{2}(\Omega))$ and $u\in U = L^{2}(0,T; \mathbb{R}^{m})$, 
where $m$ is the number of actuators. The initial datum $y_{_{0}}$ 
is in $L^{2}(\Omega)$, $\mathcal{B}: \mathbb{R}^{m}\longrightarrow L^{2}(\Omega)$ 
is the control operator, which is linear, possibly unbounded, 
and depending on the number and structure of actuators.

Several definitions and preliminary 
results will be needed to study system \eqref{sys1:eq1}. We begin
by recalling the most important function used in fractional calculus,
\emph{Euler's gamma function}, which is defined as
\begin{equation*}
\Gamma(n) = \ds\int_{0}^{\infty}t^{n-1}e^{-t}dt.
\end{equation*}
This function is a generalization of the factorial: 
if $n \in \mathbb{N}$, then $\Gamma(n) = (n-1)!$.

\begin{definition}[See, e.g., \cite{Book:IP:1999}]
The left-sided Caputo fractional derivative of order $\alpha > 0$ 
of a function $z$ is given by
\begin{equation}
\label{Cap:Der:eq2}
{}^{C}_{0}D^{\alpha}_{t}z(t) 
= 
\begin{cases}
\ds\frac{1}{\Gamma(n-\alpha)}\int_{0}^{t}(t-s)^{n-\alpha-1}
\frac{d^{n}}{d s^{n}}z(s)ds, 
\quad n-1<\alpha<n,\quad t\geq 0,\quad n\in\mathbb{N},\\[0.3cm]
\ds\frac{d^{n}z(t)}{d t^{n}},
\quad \alpha=n\in\mathbb{N}.
\end{cases}
\end{equation}
\end{definition}
The right-sided is pointwise defined. The Caputo fractional derivative 
is a sort of regulation in the time origin for the 
Riemann--Liouville fractional derivative.

\begin{definition}[See, e.g., \cite{Paper:OPA:2004,Paper:GMB:2016,Paper:GMM:2011}]
Let $z:\mathbb{R}^{+} \rightarrow \mathbb{R}$ be a continuous function 
on $\mathbb{R}^{+}$ and $\alpha>0$. Then the expressions
\begin{equation}
\label{RL:Int:eq3}
{}_{0}I^{\alpha}_{t}z(t) 
= \ds\frac{1}{\Gamma(\alpha)}
\int_{0}^{t}(t-s)^{\alpha-1}z(s)ds,\quad t>0,
\end{equation}
and 
\begin{equation}
\label{RL:Int:eq4}
{}_{t}I^{\alpha}_{T}z(t) = \ds\frac{1}{\Gamma(\alpha)}
\int_{t}^{T}(s-t)^{\alpha-1}z(s)ds,\quad t<T,
\end{equation}
are, respectively, called the left-sided and right-sided 
Riemann--Liouville integrals of order $\alpha$.
\end{definition}

\begin{definition}[See, e.g., \cite{Paper:OPA:2004,Paper:GMB:2016,Paper:GMM:2011}]
Let $z:\mathbb{R}^{+} \rightarrow \mathbb{R}$. 
The left-sided and right-sided Riemann--Liouville fractional 
derivatives of order $\alpha$ are defined by
\begin{equation}
\label{RL:Der:eq4}
{}_{0}D^{\alpha}_{t}z(t) = \ds\frac{1}{\Gamma(n-\alpha)}
\frac{d^{n}}{dt^{n}}\int_{0}^{t}(t-s)^{n-\alpha-1}z(s)ds, \quad t>0,
\end{equation}
and
\begin{equation}
\label{RL:Der:eq5}
{}_{t}D^{\alpha}_{T}z(t) = \ds\frac{1}{\Gamma(n-\alpha)}\left( 
-\frac{d}{dt}\right)^{n} \int_{t}^{T}(s-t)^{n-\alpha-1}z(s)ds, \quad t<T,
\end{equation}
where $\alpha\in(n-1,n)$, $n\in\mathbb{N}$.
\end{definition}
We will always consider solutions of \eqref{sys1:eq1} in the weak sense. 
We denote that solution by $y(x,t;u)$ and, when there will be no possible ambiguities, 
we will also use the short notation $y_{_{u}}(t)$ or $y(u)$. Hence, we denote 
by $y_{_{u}}(T)$ the mild solution of system \eqref{sys1:eq1} at the final time $T$. 

\begin{definition}[See \cite{Paper:FM:et:al:2007}]
For $t\in[0,T]$ and any given $u\in U$, a function 
$y\in L^{2}(0,T;L^{2}(\Omega))$ is said to be a mild solution 
of system \eqref{sys1:eq1} if it satisfies
\begin{equation}
\label{MildSol:eq5}
y_{_{u}}(t) = \mathcal{R}_{_{\alpha}}(t)y_{_{0}} 
+ \ds\int_{0}^{t}(t-s)^{\alpha-1}K_{_{\alpha}}(t-s)\mathcal{B}u(s)ds,
\end{equation}
where 
\begin{equation}
\label{eq6}
\mathcal{R}_{_{\alpha}}(t) = \ds\int_{0}^{\infty}
\phi_{_{\alpha}}(\theta)S(t^{\alpha}\theta)d\theta
\end{equation}
and 
\begin{equation}\label{eq7}
K_{_{\alpha}}(t) = \alpha\ds\int_{0}^{\infty}\theta
\phi_{_{\alpha}}(\theta)S(t^{\alpha}\theta)d\theta
\end{equation}
with $\phi_{_{\alpha}}(\theta)$ given by 
\begin{equation*}
\phi_{_{\alpha}}(\theta) = \ds\frac{1}{\alpha}\theta^{-1-1/\alpha}
\psi_{_{\alpha}}(\theta^{-1/\alpha}),
\end{equation*}
where $\psi_{_{\alpha}}$ is the following probability density function:
\begin{equation}
\label{eq:pdf}
\psi_{_{\alpha}}(\theta) = \ds\frac{1}{\pi} 
\sum_{n=1}^{\infty}(-1)^{n-1}\theta^{-\alpha n-1}
\frac{\Gamma(n\alpha+1)}{n!}\sin(n\pi\alpha),
\quad\theta\in(0,\infty).
\end{equation}
\end{definition}

\begin{remark}
The probability density function \eqref{eq:pdf} satisfies the following properties:
\begin{equation}\label{density:eq8}
\ds\int_{0}^{\infty}e^{-\lambda\theta}\psi_{_{\alpha}}(\theta)d\theta 
= e^{-\lambda^{\alpha}},\qquad \int_{0}^{\infty}\psi_{_{\alpha}}(\theta)d\theta 
= 1,\quad \alpha\in(0,1),
\end{equation}
and
\begin{equation}\label{gamma:func}
\ds\int_{0}^{\infty}\theta^{\nu}\phi_{_{\alpha}}(\theta)d\theta 
= \ds\frac{\Gamma(1+\nu)}{\Gamma(1+\alpha\nu)},\quad \nu\geq 0.
\end{equation}
\end{remark}

Let $H : L^{2}(0,T; \mathbb{R}^{m})\rightarrow L^{2}(\Omega)$ be 
defined as
\begin{equation}
\label{H:eq10}
Hu = \ds\int_{0}^{T}(T-s)^{\alpha-1}K_{_{\alpha}}(T-s)\mathcal{B}u(s)ds,
\quad\forall u\in L^{2}(0,T; \mathbb{R}^{m}),
\end{equation}
where $m$ is the number of actuators. We assume that $(S^{*}(t))_{_{t\geq 0}}$ 
is a strongly continuous semi-group generated by $A^{*}$ on the state space 
$L^{2}(\Omega)$. For $v\in L^{2}(\Omega)$, one has
\begin{equation}
\label{H*:eq11}
\begin{array}{lll}
\left\langle Hu , v\right\rangle 
&= \left\langle\displaystyle\int_{0}^{T}(T-s)^{\alpha-1}
K_{\alpha}(T-s)\mathcal{B}u(s)ds, v\right\rangle_{L^{2}(\Omega)}\\
& = \displaystyle\int_{0}^{T} \langle (T-s)^{\alpha-1}
K_{\alpha}(T-s)\mathcal{B}u(s),v\rangle_{L^{2}(\Omega)} ds\\
& = \displaystyle\int_{0}^{T} \left\langle u(s),
\mathcal{B}^{*}(T-s)^{\alpha-1}
K_{\alpha}^{*}(T-s)v\right\rangle_{L^{2}(0,T; \mathbb{R}^{m})} ds\\
& = \langle u, H^{*}v \rangle,
\end{array}
\end{equation}
where by $\langle\cdot,\cdot\rangle$ we denote the duality pairing 
of space $L^{2}(\Omega)$, $\mathcal{B}^{*}$ 
is the adjoint operator of $\mathcal{B}$, and
\begin{equation*}
K_{_{\alpha}}^{^{*}}(t) = \alpha\ds\int_{0}^{\infty}\theta
\phi_{_{\alpha}}(\theta)S^{*}(t^{\alpha}\theta)d\theta.
\end{equation*}
In order to prove our main results, the following lemmas are needed.

\begin{lemma}[See \cite{Book:KM:2009}]
\label{lemma1}
Let the reflection operator $\mathcal{Q}$ 
on the interval $[0,T]$ be defined as follows:
\begin{equation*}
\mathcal{Q}h(t):=h(T-t)
\end{equation*}
for some function $h$ that is
differentiable and integrable. 
Then the following relations hold:
\begin{equation*}
\mathcal{Q}{}_{0}I^{\alpha}_{t}h(t) 
= {}_{t}I^{\alpha}_{T}\mathcal{Q}h(t), 
\qquad \mathcal{Q}{}_{0}D^{\alpha}_{t}h(t) 
= {}_{t}D^{\alpha}_{T}\mathcal{Q}h(t)
\end{equation*}
and
\begin{equation*}
{}_{0}I^{\alpha}_{t}\mathcal{Q}h(t) 
= \mathcal{Q}{}_{t}I^{\alpha}_{T}h(t), 
\qquad {}_{0}D^{\alpha}_{t}\mathcal{Q}h(t) 
= \mathcal{Q}{}_{t}D^{\alpha}_{T}h(t).
\end{equation*}
\end{lemma}

\begin{lemma}[See, e.g., \cite{conf:pap:IP:YQC:2007}]
\label{lemma2}
For $t\in [a,b]$ and $n-1<\alpha<n$, $n\in\mathbb{N}$, 
the following integration by parts formula holds:
\begin{equation*}
\displaystyle\int_{a}^{b}f(t){}^{C}_{0}D^{\alpha}_{t} g(t)dt 
= \sum_{r=0}^{k-1}(-1)^{k-1-r}\left[ g^{r}(t) 
{}_{t}D^{\alpha-1-r}_{b}f(t) \right]_{t=a}^{t=b} 
+ (-1)^{k}\int_{a}^{b}g(t){}_{t}D^{\alpha}_{b}f(t)dt.
\end{equation*}
In particular, if $0<\alpha<1$, then 
\begin{equation*}
\displaystyle\int_{a}^{b}f(t) {}^{C}_{0}D^{\alpha}_{t} g(t)dt 
= \left[  g(t){}_{t}I^{1-\alpha}_{b}f(t)\right]_{t=a}^{t=b} 
+ \int_{a}^{b}g(t) {}_{t}D^{\alpha}_{b}f(t)dt.
\end{equation*}
\end{lemma}

We also recall the fractional Green's formula:

\begin{lemma}[See, e.g., \cite{Paper:GMB:2016,Paper:Bah:Tan:2018}]
\label{lema:FGF}
Let $0<\alpha\leq 1$ and $t\in[0,T]$. Then,
\begin{equation*}
\begin{array}{lll}
& \ds\int_{0}^{T}\int_{\Omega} \left( {}^{C}_{0}D^{\alpha}_{t}y(x,t)
+\mathcal{A}y(x,t) \right)\varphi(x,t) dx dt 
=  \ds\int_{0}^{T}\int_{\Omega} y(x,t) \left( {}_{t}D^{\alpha}_{T}\varphi(x,t)
+\mathcal{A^{*}}\varphi(x,t) \right)\\
&+ \ds\int_{0}^{T}\int_{\partial\Omega} y(x,t) {}_{t}I^{1-\alpha}_{T}\varphi(x,t)d\Gamma dt 
- \int_{0}^{T}\int_{\partial\Omega} y(x,t)\frac{\partial \varphi(x,t)}{\partial\nu_{\mathcal{A}}} 
+ \int_{0}^{T}\int_{\partial\Omega} 
\frac{\partial y(x,t)}{\partial\nu_{\mathcal{A}}}\varphi(x,t)d\Gamma dt
\end{array}
\end{equation*}
for any $\varphi\in C^{\infty}(\overline{Q})$. 
\end{lemma}


\section{\; Regional Enlarged Controllability}
\label{Sec:3}

We extend the definition of controllability 
first introduced by Lions in \cite{Book:JLL:1988}, 
to the case of sub-diffusion fractional systems. 
For that, we consider a nonempty sub-vectorial space 
$G\subset L^{2}(\Omega)$, which is supposed 
to be closed and convex. 

\begin{definition}
\label{def:EEC}
Given a final time $T>0$, we say that system \eqref{sys1:eq1} is exactly 
enlarged controllable (i.e., $G$-controllable) if, for every $y_{_{0}}$ 
in a suitable functional space, there exists a control $u$ such that
\begin{equation}
\label{eq:Def31}
y(\cdot,T;u)\in G.
\end{equation}
\end{definition}

\begin{remark}
Obviously, the concept of exact enlarged controllability depends on $G$.
\end{remark}

\begin{remark}
If $G=\{0\}$, then we get from Definition~\ref{def:EEC} 
the classical notion of exact controllability.
\end{remark}

\begin{remark}
Exact controllability implies exact enlarged controllability (EEC) 
for every set $G$. The inverse is, however, not true. 
\end{remark}

\begin{theorem}
System \eqref{sys1:eq1} is said to be exactly 
enlarged controllable if, and only if, 
\begin{equation}
\label{eq:Thm32}
G-\left\lbrace \mathcal{R}_{_{\alpha}}(T)y_{_{0}}\right\rbrace 
\cap Im H \neq \emptyset.
\end{equation}
\end{theorem}

\begin{proof}
Suppose that one has exact enlarged controllability (EEC) of \eqref{sys1:eq1} 
relatively to $G$, which means that $y_{_{u}}(T)\in G$. Then, 
$y_{_{u}}(T) = \mathcal{R}_{_{\alpha}}(T)y_{_{0}} + Hu$
and, denoting 
\begin{equation*}
w = y_{_{u}}(T) - \mathcal{R}_{_{\alpha}}(T)y_{_{0}} = Hu,
\end{equation*}
it follows that $w\in Im H$ and $w\in G-\left\lbrace \mathcal{R}_{_{\alpha}}(T)y_{_{0}}\right\rbrace$. 
Thus, \eqref{eq:Thm32} holds. Conversely, suppose \eqref{eq:Thm32} is true. 
Then, there exists $z\in G-\left\lbrace \mathcal{R}_{_{\alpha}}(T)y_{_{0}}\right\rbrace$ 
such that $z\in Im H$. So, there exists $u\in L^{2}(0,T; \mathbb{R}^{m})$ 
such that $z = Hu$. Hence, $z =Hu\in G\left\lbrace \mathcal{R}_{_{\alpha}}(T)y_{_{0}}\right\rbrace$, 
$Hu + \mathcal{R}_{_{\alpha}}(T)y_{_{0}} = y_{_{u}}(T) \in G$, and we have EEC relatively to $G$.
\end{proof}

We recall that an actuator is defined by a couple $(D,f)$, 
where $D$ is a nonempty closed part of $\overline{\Omega}$, 
which represents the geometric support of the actuator, 
and $f\in L^{2}(D)$, which defines the spacial distribution 
of the action on the support $D$. In the case of a pointwise actuator, 
$D=\{b\}$ and $f=\delta(b-\cdot)$, where $\delta$ is the Dirac mass concentrated in $b$. 
For more details on actuators, we refer the interested reader to 
\cite{Book:AEJ:AJP:1988,Paper:Zerrik:et:al:actua:2000}.

\begin{definition}
The actuator $(D,f)$ is said to be $G$-strategic if one has 
exact enlarged controllability relatively to $G$. 
\end{definition}


\section{\; Extended RHUM Approach}
\label{Sec:4}

Now we extend the RHUM introduced by Lions 
in \cite{Book:JLL:HUM:1971,Paper:JLL:HUM:1988} to the fractional-order case. 
The aim is to find the control steering the system \eqref{sys1:eq1} 
from the initial state $y_{_{0}}$ into the functional subspace $G$.
Let us denote by $G^{\circ}$ the polar space of $G$. Hence,
\begin{equation*}
\varphi_{_{0}}\in G^{\circ} \Longleftrightarrow 
\langle\varphi_{_{0}},\phi\rangle = 0\quad \forall\phi\in G,
\end{equation*}
where $\langle\cdot,\cdot\rangle$ denotes the scalar product in $L^{2}(\Omega)$. 
Let us also denote by $\mathcal{A}^{*}$ the adjoint operator of $\mathcal{A}$ 
and, for any $\varphi_{_{0}}\in G^{\circ}$, consider the following adjoint system:
\begin{equation}
\label{adjoint:sys}
\begin{cases}
\ds {}_{t}D^{\alpha}_{T}Q\varphi(t) = -\mathcal{A}^{*}Q\varphi(t),\\
\ds\lim_{t\rightarrow T^{-}}{}_{t}I^{1-\alpha}_{T}Q\varphi(t) 
= \varphi_{_{0}} \in D(\mathcal{A}^{*})\subseteq L^{2}(\Omega).
\end{cases}
\end{equation}
It follows from Lemma~\ref{lemma2} that \eqref{adjoint:sys} 
can be rewritten as
\begin{equation}
\label{sys:ss:Q}
\begin{cases}
\ds {}_{0}D^{\alpha}_{t}\varphi(t) = -\mathcal{A}^{*}\varphi(t),\\
\ds\lim_{t\rightarrow 0^{+}}{}_{0}I^{1-\alpha}_{t}\varphi(t) 
= \varphi_{_{0}} \in D(\mathcal{A}^{*})\subseteq L^{2}(\Omega),
\end{cases}
\end{equation}
with solution given by 
$\varphi(t) = -t^{\alpha-1}K_{\alpha}^{*}(t)\varphi_{_{0}}$.


\subsection{\; Excitation of the system with a zone actuator}
\label{zone:case}

We consider system \eqref{sys1:eq1} excited by a zone actuator 
$\mathcal{B}u(t) = \chi_{_{D}}f(x)u(t)$. Then the system 
is written as follows:
\begin{equation}
\label{sys:zone:act}
\begin{cases}
\ds {}^{C}_{0}D^{\alpha}_{t}y(t) 
= \mathcal{A}y(t)  + \chi_{_{D}}f(x)u(t), \quad & t\in [0,T],\\
y(0)=y_{_{0}}\in D(\mathcal{A}).
\end{cases}
\end{equation}
Let $w_{i}(x)$ denote the eigenfunctions of operator $\mathcal{A}$ 
associated with the eigenvalues $\lambda_{i}$. For any $\varphi_{_{0}}\in G^{\circ}$, 
we define the following semi-norm on $G^{\circ}$:
\begin{equation}
\label{norm:eq}
\|\varphi_{_{0}}\|_{_{G^{\circ}}}^{^{2}} 
:= \ds\int_{0}^{T}\langle f,\varphi(t)\rangle_{_{L^{2}(D)}}^{^{2}} dt.
\end{equation}

\begin{theorem}
\label{Thm:norm:EEC}
The	semi-norm \eqref{norm:eq} defines a norm on $G^{\circ}$ 
if $\langle w_{i},f\rangle_{_{L^{2}(D)}} \neq 0$. In that case, 
we have exact enlarged controllability relatively to $G$.
\end{theorem}

\begin{proof}
We consider the following problem: 
\begin{equation}\label{sys:zone:act2}
\begin{cases}
\ds {}^{C}_{0}D^{\alpha}_{t}\Psi(t) 
= \mathcal{A}\Psi(t)  + \chi_{_{D}}f(x)u(t), \quad & t\in [0,T],\\
\Psi(0)=y_{_{0}}\in D(\mathcal{A}).
\end{cases}
\end{equation}
The solution $\Psi: [0,T]\rightarrow L^{2}(\Omega)$ of 
\eqref{sys:zone:act2} is continuous. If we can find 
$\varphi_{_{0}}\in G^{\circ}$ such that 
\begin{equation}
\label{EEC:Cond}
\Psi(T) \in G,
\end{equation}
then $u = \langle f,\varphi(t)\rangle_{_{L^{2}(D)}}$ is the control 
that ensures the exact enlarged controllability relatively to $G$ and 
\begin{equation*}
y(u) = \Psi.
\end{equation*}
To explain \eqref{EEC:Cond}, it is necessary to introduce 
the orthogonal projection $\mathcal{P}$ on the orthogonal of $G$, 
denoted by $G^{\bot}$. Let us define the affine operator 
$M: G^{\circ}\rightarrow G^{\bot}$ such that
\begin{equation}
\label{App:Affine}
M\varphi_{_{0}} = \mathcal{P}(\Psi(T)).
\end{equation}
Then we need to solve equation 
\begin{equation}
\label{Eq:null}
M\varphi_{_{0}} = 0.
\end{equation}
For that, we decompose $M$ into two parts:
a linear and a constant one. 
Let $\Psi_{_{0}}$ be solution of
\begin{equation}
\label{linear:eq}
\begin{cases}
\ds {}^{C}_{0}D^{\alpha}_{t}\Psi_{_{0}}(t) 
= \mathcal{A}\Psi_{_{0}}(t)  
+ \chi_{_{D}}f(x)u(t), \quad & t\in [0,T],\\
\Psi_{_{0}}(0)=0,
\end{cases}
\end{equation}
and $\Psi_{_{1}}$ solution of
\begin{equation}
\label{constant:eq}
\begin{cases}
\ds {}^{C}_{0}D^{\alpha}_{t}\Psi_{_{1}}(t) 
= \mathcal{A}\Psi_{_{1}}(t),   \quad & t\in [0,T],\\
\Psi_{_{1}}(0)=y_{_{0}}\in D(\mathcal{A}).
\end{cases}
\end{equation}
Then, 
\begin{equation}
\label{M:Decomposition}
M\varphi_{_{0}} = \mathcal{P}(\Psi_{_{0}}(T)) 
+ \mathcal{P}(\Psi_{_{1}}(T)),
\end{equation}
where we set $M_{0}\varphi_{_{0}} = \mathcal{P}(\Psi_{_{0}}(T))$ 
with $M_{0}\in\mathcal{L}(G^{\circ},G^{\bot})$. From \eqref{Eq:null} 
and \eqref{M:Decomposition}, all consists now to solve
\begin{equation}
\label{main:eq}
M_{0}\varphi_{_{0}} = -\mathcal{P}(\Psi_{_{1}}(T)).
\end{equation}
For that, we compute the scalar product
\begin{equation}
\mu = \langle M_{0}\varphi_{_{0}},\varphi_{_{0}}\rangle,
\qquad\varphi_{_{0}}\in G^{\circ},
\end{equation}
where $\langle\cdot,\cdot\rangle$ is the dual pairing of $G^{\bot}$ and $G^{\circ}$. 
By definition, 
\begin{equation}
\label{projection:eq}
\langle\mathcal{P}(\tilde{g}),\bar{g}\rangle = 0,
\qquad \forall\bar{g}\in G^{\circ}.
\end{equation}
Using \eqref{projection:eq}, we have
\begin{equation}
\label{mu:eq}
\mu = \langle\Psi_{_{0}},\varphi_{_{0}}\rangle
\end{equation}
for $\bar{g}=\varphi_{_{0}}$, $\tilde{g}=\Psi_{_{0}}(T)$.
To compute the last expression \eqref{mu:eq}, we multiply system 
\eqref{linear:eq} by $\varphi$, integrating over 
$Q = \Omega\times[0,T]$. We obtain that
\begin{equation*}
\ds\int_{0}^{T}\int_{\Omega} {}^{C}_{0}D^{\alpha}_{t}\Psi_{_{0}}(t)\varphi(t) dxdt 
- \int_{0}^{T}\int_{\Omega} \mathcal{A}\Psi_{_{0}}(t)\varphi(t) dxdt 
= \int_{0}^{T}\int_{\Omega} \chi_{_{D}}f(x)u(t)\varphi(t) dx dt. 
\end{equation*}
Using Lemma~\ref{lema:FGF} (fractional Green's formula), we have
\begin{multline}
\label{Green:formula}
-\ds\int_{0}^{T}\int_{\Omega}\mathcal{A}\Psi_{_{0}}(t)\varphi(t) dxdt 
= -\int_{0}^{T}\int_{\partial\Omega}\frac{\partial 
\Psi_{_{0}}(t)}{\partial\nu_{\mathcal{A}}} \varphi(t) d\sigma dt\\
+ \int_{0}^{T}\int_{\partial\Omega}\Psi_{_{0}}(t)
\frac{\partial\varphi(t)}{\partial\nu_{\mathcal{A}}} d\sigma dt 
- \int_{0}^{T}\int_{\Omega}\Psi_{_{0}}\mathcal{A^{*}}\varphi(t)dxdt
\end{multline}
and 
\begin{equation}
\label{Caputo:Transf}
\begin{array}{ll}
\ds\int_{0}^{T}\int_{\Omega}{}^{C}_{0}D^{\alpha}_{t}\Psi_{_{0}}(t)\varphi(t) dxdt 
&= \ds\int_{0}^{T}\int_{\Omega} \Psi_{_{0}}(t){}^{C}_{0}D^{\alpha}_{t}\varphi(t) dx dt 
+ \int_{\partial\Omega} \Psi_{_{0}}(T)
\lim_{t\rightarrow T}{}_{t}I^{1-\alpha}_{T}\varphi(T) d\sigma\\[0.3cm]
&- \ds\int_{\partial\Omega} \Psi_{_{0}}(0)
\lim_{t\rightarrow T}{}_{t}I^{1-\alpha}_{T}\varphi(0) d\sigma.
\end{array}
\end{equation}
From \eqref{Green:formula} and \eqref{Caputo:Transf}, it follows that
\begin{equation*}
\langle M_{0}\varphi_{_{0}},\varphi_{_{0}}\rangle 
= \int_{0}^{T} \left( \langle f(x),\varphi(t)\rangle_{_{L^{2}(D)}}\right)^{2} dt 
= \|\varphi_{_{0}}\|^{^{2}}_{G^{\circ}}.
\end{equation*}
Hence,
\begin{equation}
\label{mu:semi:norm}
\mu = \int_{0}^{T} \left( \langle f(x),\varphi(t)\rangle_{_{L^{2}(D)}}\right)^{2} dt.
\end{equation}
The essential point now is that the previous formula \eqref{mu:semi:norm} 
is a semi-norm on $G^{\circ}$. We prove that if 
$\langle w_{i},f\rangle_{_{L^{2}(D)}} \neq 0$, then the mapping \eqref{norm:eq} 
is a norm, which is equivalent to the norm of $G^{\circ}$.
The mapping \eqref{norm:eq} is a norm on $G^{\circ}$. Indeed,
\begin{equation*}
\|\varphi_{_{0}}\|_{_{G^{\circ}}} = 0 \Longleftrightarrow \langle 
f,\varphi(t)\rangle_{_{L^{2}(D)}}^{^{2}} = 0,
\end{equation*}
which is equivalent to
\begin{equation}
\label{eq:32}
-\ds\sum_{i=1}^{\infty} t^{\alpha-1}\alpha
\int_{0}^{\infty}\theta\phi_{_{\alpha}}(\theta) e^{\lambda_{i}(t^{\alpha}\theta)}
d\theta\langle f,w_{i}\rangle \langle \varphi_{_{0}},w_{i}\rangle = 0.
\end{equation}
Thus, \eqref{eq:32} gives 
\begin{equation*}
\langle f,w_{i}\rangle \langle 
\varphi_{_{0}},w_{i}\rangle = 0.
\end{equation*}
Using the assumption that $\langle f,w_{i}\rangle \neq 0$, 
we deduce that $ \langle \varphi_{_{0}},w_{i}\rangle = 0$. Therefore, 
$\varphi_{_{0}} = 0$, \eqref{norm:eq} defines a norm on $G^{\circ}$, 
and $\mu$ is an isomorphism from $G^{\circ}$ to $G^{\bot}$. 
Moreover, equation \eqref{main:eq} admits a unique solution.
\end{proof}


\subsection{\; Excitation of the system with a pointwise actuator}

Now we consider system \eqref{sys1:eq1} excited by a pointwise actuator. 
In this case the control is of type $Bu(t) = \delta(x-b)u(t)$, 
where $b\in\Omega$ refers to the location of the actuator and $u\in U$. 
Hence, system \eqref{sys1:eq1} is written as follows:
\begin{equation}
\label{sys:pointwise:act}
\begin{cases}
\ds {}^{C}_{0}D^{\alpha}_{t}y(t) 
= \mathcal{A}y(t)  + \delta(x-b)u(t), \quad & t\in [0,T], \\
y(0)=y_{_{0}}\in D(\mathcal{A}).
\end{cases}
\end{equation}
For $\varphi_{_{0}}\in G^{\circ}$, we consider the adjoint system
\begin{equation}
\label{adjoint:sys:2}
\begin{cases}
\ds {}_{t}D^{\alpha}_{T}Q\varphi(t) 
= -\mathcal{A}^{*}Q\varphi(t), \\
\ds\lim_{t\rightarrow T^{-}}{}_{t}I^{1-\alpha}_{T}Q\varphi(t) 
= \varphi_{_{0}} \in D(\mathcal{A}^{*})\subseteq L^{2}(\Omega),
\end{cases}
\end{equation}
and the mapping
\begin{equation}
\label{norm:eq:pointwise:case}
\|\varphi_{_{0}}\|_{_{G^{\circ}}}^{^{2}} 
:= \ds\int_{0}^{T} \varphi^{2}(b,t)dt,
\end{equation}
which defines a semi-norm on $G^{\circ}$. 
Let us consider the system
\begin{equation}
\label{sys:point:act}
\begin{cases}
\ds {}^{C}_{0}D^{\alpha}_{t}\Phi(t) =
 \mathcal{A}\Phi(t)  + \delta(x-b)u(t), \quad & t\in [0,T],\\
\Phi(0)=y_{_{0}}\in D(\mathcal{A}),
\end{cases}
\end{equation}
and the operator $M: G^{\circ}\rightarrow G^{\bot}$ defined by
\begin{equation*}
M\varphi_{_{0}} = \mathcal{P}(\Phi(T)),
\end{equation*}
where we write $M$ as
\begin{equation*}
M\varphi_{_{0}} 
= \mathcal{P}(\Phi_{0}(T)+\Phi_{1}(T))
\end{equation*}
with $\Phi_{0}$ and $\Phi_{1}$ solutions of the systems
\begin{equation}
\label{linear:eq:point:case}
\begin{cases}
\ds {}^{C}_{0}D^{\alpha}_{t}\Phi_{_{0}}(t) 
= \mathcal{A}\Phi_{_{0}}(t)  + \delta(x-b)u(t), \quad & t\in [0,T], \\
\Phi_{_{0}}(0)=0,
\end{cases}
\end{equation}
and
\begin{equation}
\label{constant:eq:point:case}
\begin{cases}
\ds {}^{C}_{0}D^{\alpha}_{t}\Phi_{_{1}}(t) 
= \mathcal{A}\Phi_{_{1}}(t), \quad & t\in [0,T],\\
\Phi_{_{1}}(0)=y_{_{0}}\in D(\mathcal{A}),
\end{cases}
\end{equation}
respectively. Let us set $M_{0}\varphi_{_{0}} = \mathcal{P}(\Phi_{_{0}}(T))$ 
with $M_{0}\in\mathcal{L}(G^{\circ},G^{\bot})$. Then, 
all returns to solve
\begin{equation}
\label{main:eq:point:case}
M_{0}\varphi_{_{0}} = -\mathcal{P}(\Phi_{_{1}}(T)).
\end{equation}
Similar arguments as the ones used in Section~\ref{zone:case}, 
allow us to prove the following result:

\begin{theorem}\label{theo3}
If $w_{i}(b) \neq 0$, then the mapping \eqref{norm:eq:pointwise:case} 
defines a norm on $G^{\circ}$ and one has exact enlarged controllability 
(EEC) relatively to $G$. Moreover, the control
\begin{equation*}
u = \varphi(b,t)
\end{equation*}
ensures the EEC into $G$.
\end{theorem}


\section{\; Fractional Optimal Control}
\label{Sec:5}

In this section we are concerned with the following optimization problem:
\begin{equation}
\label{min:pb}
\begin{cases}
\inf \mathcal{J}(u),\\
u\in U_{ad},
\end{cases}
\end{equation}
where 
$$
\mathcal{J}(u) = \ds\frac{1}{2}\int_{0}^{T}\|u\|_{_{U}}^{^{2}}dt
$$ 
and the feasible set $U_{ad} = \{u\in U \;|\; y_{_{u}}(T)\in G\}$
is assumed to be non-empty.

\begin{theorem}
\label{Main:Thm}
Assume that one has exact enlarged controllability relatively to $G$.
Then, the optimal control problem \eqref{min:pb} has a unique solution 
given by $u^{*}(t) = \langle f,\varphi(t)\rangle$,
in case of a zone actuator, and
$u^{*}(t) = \varphi(b,t)$, in case of a pointwise actuator.
Such control ensures the transfer of system \eqref{sys1:eq1}
into $G$ with a minimum energy cost in the sense of $\mathcal{J}$. 
\end{theorem}

\begin{proof}
Suppose that we have exact enlarged controllability relatively to $G$. 
Then we set $\epsilon > 0$ and we consider the following problem:
\begin{equation}
\label{min:pb:penal}
\mathcal{J}_{\epsilon}(u,z) = \ds\frac{1}{2}\int_{0}^{T}u^{2}(t)dt
+ \frac{1}{2\epsilon}\int_{Q}\left( {}^{C}_{0}D^{\alpha}_{t}z(t)
-\mathcal{A}z(t)-\chi_{_{D}}f(x)u(t)\right) ^{2}dQ,
\end{equation}
where
\begin{equation}
\label{sys:penal}
\begin{cases}
{}^{C}_{0}D^{\alpha}_{t}z(t) - \mathcal{A}z(t) 
- \chi_{_{D}}f(x)u(t)\:\in L^{2}(Q), \\
z(0)=z_{0}\in D(\mathcal{A}),  \\
z_{u}(T)\in G.
\end{cases}
\end{equation}
The set of pairs $(u,z)$ that verify \eqref{sys:penal},
denoted by $W$, is nonempty, and we consider problem
\begin{equation}
\label{min:pb:S}
\begin{cases}
\inf \mathcal{J}_{\epsilon}(u,z),\\
(u,z)\in W.
\end{cases}
\end{equation}
Let $\{u_{\epsilon},z_{\epsilon}\}$ be solution of \eqref{min:pb:S}. Then, 
\begin{equation}
0 < \mathcal{J}_{\epsilon}(u_{\epsilon},z_{\epsilon}) 
= \inf \mathcal{J}_{\epsilon}(u,z) 
< \inf \mathcal{J}_{\epsilon}(u) < \infty,  \quad u\in U_{ad},
\end{equation}
where $\mathcal{J}_{\epsilon}(u) = \ds\frac{1}{2}\int_{0}^{T}u^{2}(t)dt$. 
Tending $\epsilon$ to 0, we conclude that
\begin{eqnarray}
\label{cond:1}
\begin{cases}
\|u_{\epsilon}\|\leq C,\\
\|{}^{C}_{0}D^{\alpha}_{t}z(x,t) - \mathcal{A}z(x,t) 
- \chi_{_{D}}f(x)u(t)\|\leq C\sqrt{\epsilon},
\end{cases}
\end{eqnarray}
where $C$ represents different positive 
constants independent of $\epsilon$. 
It follows from \eqref{cond:1} that
\begin{equation*}
\|{}^{C}_{0}D^{\alpha}_{t}z(x,t) - \mathcal{A}z(x,t)\|
\leq C(1+\sqrt{\epsilon}).
\end{equation*}  
Hence, when $\epsilon\rightarrow 0$, we have that $u_{\epsilon}$ 
is bounded and we can extract a sequence such that
\begin{equation*}
\begin{array}{cc}
u_{\epsilon}\rightharpoonup\tilde{u}
\quad &\mbox{weakly in} \quad U,\\
z_{\epsilon}\rightharpoonup z
\quad &\mbox{weakly in} \quad L^{2}(Q).
\end{array}
\end{equation*}
By the semi-continuity of $\mathcal{J}$, one has
\begin{equation*}
\mathcal{J}(u^{*}) \leq \lim\inf\mathcal{J}_{\epsilon}(u_{\epsilon}) 
\leq \lim\inf\mathcal{J}_{\epsilon}(u_{\epsilon},z_{\epsilon}).
\end{equation*}
Therefore,
\begin{equation*}
\mathcal{J}(u^{*}) = \inf\mathcal{J}(u), \quad u \in U_{ad},
\end{equation*}
and 
\begin{equation*}
u^{*} = \tilde{u}.
\end{equation*}
Define
\begin{equation*}
p_{\epsilon} = -\ds\frac{1}{\epsilon}\left( 
{}^{C}_{0}D^{\alpha}_{t}z_{\epsilon}(x,t)
-\mathcal{A}z_{\epsilon}(x,t)
-\chi_{_{D}}f(x)u_{\epsilon}(t)\right).
\end{equation*}
The Euler equation relatively to problem \eqref{min:pb:S} is given by
\begin{equation*}
\ds\int_{0}^{T}u_{\epsilon}(t)u(t)dt - \int_{0}^{T} \langle p_{\epsilon},
{}^{C}_{0}D^{\alpha}_{t}\eta(t)-\mathcal{A}\eta(t)\rangle dt 
= \int_{0}^{T}\langle p_{\epsilon},f\rangle u(t) dt
\end{equation*}
with $u\in U_{ad}$ and $\eta$ such that
\begin{equation*}
\begin{cases}
{}^{C}_{0}D^{\alpha}_{t}\eta(t) - \mathcal{A}\eta(t)  
= \chi_{_{D}}f(x)u(t) \quad &\mbox{in}\quad Q, \\
\eta(0)=0  \quad &\mbox{on}\quad\Omega, \\
\eta(T)\in G.
\end{cases}
\end{equation*}
We deduce that $p_{\epsilon}$ satisfies
\begin{equation*}
\begin{cases}
{}^{C}_{0}D^{\alpha}_{t}p_{\epsilon}(t) - \mathcal{A}p_{\epsilon}(t)  
= \chi_{_{D}}f(x)\langle p_{\epsilon},f\rangle_{_{L^{2}(D)}} 
\quad &\mbox{in}\quad Q, \\
p_{\epsilon}(0)=0  \quad &\mbox{on}\quad\Omega, 
\end{cases}
\end{equation*}
and $\langle \eta(T),p_{\epsilon}(T)\rangle = 0$ for all $\eta$ 
with $\eta(T)\in G$. Then, $p_{\epsilon}\in G^{\circ}$. 
If we suppose that 
\begin{equation*}
\int_{0}^{T}\langle p_{\epsilon},f\rangle^{^{2}} dt 
\geq C\|p_{\epsilon}(T)\|_{H_{0}^{1}(\Omega)}^{^{2}},
\end{equation*}
then we can switch to the limit when $\epsilon$ tends to 0. 
Moreover, because we have exact enlarge controllability relatively 
to $G$, we obtain the following optimality problem:
\begin{equation*}
\begin{cases}
{}^{C}_{0}D^{\alpha}_{t}z(t) - \mathcal{A}z(t)  
= \chi_{_{D}}f(x)u(t) \quad &\mbox{in}\quad Q, \\
z(0)=z_{_{0}}(x)  \quad &\mbox{on}\quad\Omega, \\
{}^{C}_{0}D^{\alpha}_{t}p(t) - \mathcal{A}p(t) 
= \chi_{_{D}}f(x)\langle p,f\rangle_{_{L^{2}(D)}} 
\quad &\mbox{in}\quad Q, \\
p(0)=0  \quad &\mbox{on}\quad\Omega,\\
p(T)\in G^{\circ}.
\end{cases}
\end{equation*}
Thus, we take $p(T)\in G^{\circ}$ and we introduce the  
solution $\varphi$ of \eqref{sys:ss:Q}. Then, $\psi = z$ 
if $\psi(T)\in G$, which proves that \eqref{main:eq} 
has a unique solution for $\varphi_{_{0}}\in G^{\circ}$.
\end{proof}


\section{\; Examples} 
\label{Sec:examples}

We give two examples, illustrating the obtained results.

\subsection{Example 1: case of a zonal actuator}

Let us consider the following time fractional differential equation 
with a zonal actuator: $Bu(t) = \chi_{_{[a,b]}}u(t)$, $0\leq a\leq b\leq 1$,
\begin{equation}
\label{sys:exp1}
\begin{cases}
{}^{C}_{0}D^{0.4}_{t}z(t) = \Delta z(t)  
+ \chi_{_{[a,b]}}u(t) \quad & \quad [0,1]\times [0,T], \\
z(x,0)=z_{_{0}}(x)  \quad &\quad [0,1], \\
z(0,t) = z(1,t) = 0\quad & \quad [0,T].
\end{cases}
\end{equation}
Here the state space is $L^{2}(0,1)$. Since the operator 
$\mathcal{A} = \Delta = \ds\frac{\partial^{2}\cdot}{\partial x^{2}}$ 
generates a compact, analytic, self-adjoint $C_{_{0}}$-semigroup, 
we have $\mathcal{A} = \Delta = \ds\frac{\partial^{2}\cdot}{\partial x^{2}}$ and 
\begin{equation*}
S(t)z(x) = \ds\sum_{i=1}^{+\infty}e^{\lambda_{i}t}(z,w_{i})_{L^{2}(0,1)}w_{i}(x),
\end{equation*}
where $\lambda_{i} = -i^{2}\pi^{2}$ and $w_{i}(x) = \sqrt{2}\sin(i\pi x)$.
Moreover,
\begin{equation*}
\begin{array}{ll}
K_{_{0.4}}(t)z(x) &= 0.4 \ds\int_{0}^{\infty} 
\theta \phi_{_{0.4}}(\theta)S(t^{0.4}\theta)z d\theta\\
&= 0.4 \ds\int_{0}^{\infty}\theta\phi_{_{0.4}}(\theta)\sum_{i=1}^{\infty}
e^{\lambda_{i}t^{0.4}\theta}(z,w_{i})_{L^{2}(0,1)}w_{i}(x)d\theta\\
&= 0.4 \ds\sum_{i=1}^{\infty}(z,w_{i})_{L^{2}(0,1)}w_{i}(x)\int_{0}^{\infty}
\theta\phi_{_{0.4}}(\theta)e^{\lambda_{i}t^{0.4}\theta}d\theta.
\end{array}
\end{equation*}
It follows from \eqref{gamma:func}, and the Taylor expansion 
of the exponential function, that
\begin{equation*}
\begin{array}{ll}
K_{_{0.4}}(t)z(x) 
&= 0.4\ds\sum_{i=1}^{\infty}(z,w_{i})_{L^{2}(0,1)}w_{i}(x)
\sum_{j=0}^{\infty}\int_{0}^{\infty}\frac{\left( \lambda_{i}
t^{0.4}\right)^{j}}{j!}\theta^{j+1}\phi_{_{0.4}}(\theta)d\theta\\
&= \ds\sum_{i=1}^{\infty} (z,w_{i})_{L^{2}(0,1)}w_{i}(x)
\sum_{j=0}^{\infty} \frac{0.4 (j+1)\left( \lambda_{i}
t^{0.4}\right)^{j}}{\Gamma(1+0.4j+0.4)}\\
&= \ds\sum_{i=1}^{\infty} E_{_{0.4,0.4}}(\lambda_{i}
t^{0.4})(z,w_{i})_{L^{2}(0,1)}w_{i}(x),
\end{array}
\end{equation*}
where $E_{_{p,q}}(z) := \ds\sum_{i=0}\frac{z^{i}}{\Gamma(pi+q)}$, 
$Re(p)>0$, $q,z\in\mathbb{C}$, is the generalized Mittag--Leffler 
function (see \cite{Ref:MitagLeff}). Similarly, we have:
\begin{equation*}
\begin{array}{ll}
\mathcal{R}_{_{0.4}}(t)z(x) &= \ds\int_{0}^{\infty}
\phi_{_{0.4}}(\theta)S(t^{0.4}\theta)zd\theta\\
&= \ds\sum_{i=0}^{\infty}(z,w_{_{i}})_{L^{2}(0,1)}
E_{_{0.4,1}}(\lambda_{i}t^{0.4})w_{_{i}}(x).
\end{array}
\end{equation*}
Since the operator $\Delta$ generates a compact, 
analytic, self-adjoint and continuous semigroup, it follows that
\begin{equation*}
\begin{array}{ll}
\left( H^{*}z\right) (t) 
&= \mathcal{B}^{*}(T-t)^{-0.6}K_{_{0.4}}^{*}(T-t)z(t)\\
&= \mathcal{B}^{*}(T-t)^{-0.6}\ds\sum_{i=1}^{\infty}E_{_{0.4,0.4}}\left( \lambda_{i}(T-t)^{0.4}\right)(z,w_{_{i}})_{L^{2}(0,1)}w_{_{I}}(x) \\
&= (T-t)^{-0.6}\ds\sum_{i=1}^{\infty}E_{_{0.4,0.4}}\left( \lambda_{i}(T-t)^{0.4}\right) (z,w_{_{i}})_{L^{2}(0,1)}\int_{a}^{b}w_{_{i}}(x)dx.\\
&= (T-t)^{-0.6}\ds\sum_{i=1}^{\infty}E_{_{0.4,0.4}}\left( \lambda_{i}(T-t)^{0.4}\right) (z,w_{_{i}})_{L^{2}(0,1)}\frac{\sqrt{2}}{i\pi}\left[ \cos (i\pi x)\right]_{a}^{b}\\
&=  (T-t)^{-0.6}\ds\sum_{i=1}^{\infty}E_{_{0.4,0.4}}\left( \lambda_{i}(T-t)^{0.4}\right) (z,w_{_{i}})_{L^{2}(0,1)}\frac{\sqrt{2}}{i\pi}\sin \frac{i\pi(a+b)}{2} \sin \frac{i\pi(a-b)}{2}.
\end{array}
\end{equation*}
Moreover, by Theorem~\ref{Thm:norm:EEC}, we get that if system 
\eqref{sys:exp1} is enlarged controllable, then 
\begin{equation*}
\begin{array}{ll}
\varphi_{_{0}} \rightarrow \|\varphi_{_{0}}\|_{(L^{2}(0,1))^{*}} 
&= \ds\sum_{i=1}^{\infty} t^{-0.6} 0.4
\int_{0}^{\infty}\theta\phi_{_{0.4}}(\theta) e^{\lambda_{i}(t^{0.4}\theta)}
d\theta\langle f,w_{i}\rangle \langle \varphi_{_{0}},w_{i}\rangle \\
&= t^{-0.6}K_{0.4}(t)\langle \varphi_{_{0}},w_{_{i}}\rangle\\
&= t^{-0.6} \ds\sum_{i=1}^{\infty} E_{_{0.4,0.4}}(\lambda_{i}
t^{0.4})(z,w_{i})_{L^{2}(0,1)}w_{i}(x) \langle \varphi_{_{0}},w_{_{i}}\rangle
\end{array}
\end{equation*}
defines a norm on $(L^{2}(0,1))^{*}$. We find that the control given by
\begin{equation*}
u^{*}(t) = t^{-0.6}\ds\sum_{i=0}^{+\infty}E_{_{0.4,0.4}}(\lambda_{i}
t^{0.4})(z,w_{i})_{L^{2}(0,1)}\langle\varphi_{_{0}},w_{_{i}}\rangle
\end{equation*}
steers system \eqref{sys:exp1} to $L^{2}(0,1)$ at time $T$.


\subsection{Example 2: case of a pointwise actuator}

We now consider the following system with a pointwise control 
$Bu(t) = u(t)\delta(x-b)$, $0<b<1$:
\begin{equation}
\label{sys:exp2}
\begin{cases}
{}^{C}_{0}D^{0.4}_{t}z(t) = \Delta z(t)  
+ u(t)\delta(x-b) \quad & \quad [0,1]\times [0,T], \\
z(x,0)=z_{_{0}}(x)=0  \quad &\quad [0,1], \\
z(0,t) = z(1,t) = 0\quad & \quad [0,T].
\end{cases}
\end{equation}
Let the position of the actuator be $b=1/3$.  
Similarly to the first example, we have:
\begin{equation*}
\begin{array}{ll}
\lambda_{i} = -i^{2}\pi^{2},
\qquad w_{i}(x) = \sqrt{2}\sin(i\pi x), \quad x\in[0,1],\\
S(t)z(x) = \ds\sum_{i=1}^{+\infty} 
e^{\lambda_{i}t}\left(z,w_{i}\right)_{_{L^{2}(0,1)}} w_{i}(x),\\
\mbox{and }\,\, \ K_{_{0.4}}(t)z(x) = \ds
\sum_{i=1}^{+\infty} E_{_{0.4,0.4}}(\lambda_{i}
t^{0.4})\left( z,w_{i}\right)_{_{L^{2}(0,1)}}w_{i}(x). 
\end{array}
\end{equation*}
Moreover, by Theorem~\ref{theo3}, we get that if system 
\eqref{sys:exp2} is enlarged controllable, then
\begin{equation*}
\begin{array}{ll}
\varphi_{_{0}} \rightarrow \|\varphi_{_{0}}\|_{(L^{2}(0,1))^{*}} 
&= \ds\int_{0}^{T}\left\| (T-s)^{-0.6}K_{0.4}^{*}(T-s)\varphi_{_{0}}(b)\right\|^{2}ds \\
&= \ds\int_{0}^{T}\left\| (T-s)^{-0.6}\sum_{i=0}^{+\infty}
E_{_{0.4,0.4}}(\lambda_{i}(T-s)^{0.4})(z,w_{i})_{L^{2}(0,1)}
\varphi_{_{0}}(b)\right\|^{2}ds 
\end{array}
\end{equation*}
defines a norm on $(L^{2}(0,1))^{*}$. We also have that
$M\varphi_{_{0}} = \mathcal{P}(\varphi_{_{1}}(T))$ is 
an affine operator from $(L^{2}(0,1))^{*}$ to $(L^{2}(0,1))$, 
where $\varphi_{_{1}}(T)$ is the solution of the following system:
\begin{equation}
\begin{cases}
{}^{C}_{0}D^{0.4}_{t}\varphi_{_{1}}(t) 
= \Delta \varphi_{_{1}}(t)  
+ (T-t)^{-0.6}K_{_{0.4}}^{*}(T-t)\varphi_{_{0}}(b), \\
\varphi_{_{1}}(0)=0, \\
\varphi_{_{1}}(0,t) = \varphi_{_{1}}(1,t) = 0.
\end{cases}
\end{equation}
Then, by Theorem~\ref{Main:Thm}, we find that the control given by
\begin{equation*}
u^{*}(t) = (T-t)^{-0.6}\ds\sum_{i=0}^{+\infty}
E_{_{0.4,0.4}}(\lambda_{i}(T-s)^{0.4})(z,w_{i})_{L^{2}(0,1)}
\varphi_{_{0}}(b)
\end{equation*}
steers system \eqref{sys:exp2} to $L^{2}(0,1)$ at time $T$, 
where $\varphi_{_{0}}$ is the solution of 
\begin{equation}
\label{exp:eq:point:cas}
M_{0}\varphi_{_{0}} = -\mathcal{P}(\Phi_{_{1}}(T))
\end{equation}
and $\Phi_{_{1}}(t)$ solves
\begin{equation*}
\begin{cases}
\ds {}^{C}_{0}D^{0.4}_{t}\Phi_{_{1}}(t) 
= \Delta\Phi_{_{1}}(t), \quad & t\in [0,T],\\
\Phi_{_{1}}(0)=z_{_{0}}(x)= 0\in D(\mathcal{A}),\\
\Phi_{_{1}}(0,t) = \Phi_{_{1}}(1,t) = 0.
\end{cases}
\end{equation*}
Moreover, $u^{*}$ is the solution of the minimum problem \eqref{min:pb}.


\section{\; Conclusion}
\label{Sec:6}

In this work we considered fractional diffusion equations 
in the sense of Caputo. We studied exact enlarged controllability
for such control systems by using an extended RHUM (Reverse Hilbert Uniqueness Method) 
approach and a penalization technique, covering both zone 
and pointwise actuators. The optimal control of
a minimum energy problem has been characterized explicitly.
The two methods complement each other: using the RHUM approach, 
we computed the control steering the system, for both cases 
of zone and pointwise actuators; by using the penalization method, 
we proved that such control is unique.
We claim that our techniques and results can be adapted
to cover boundary conditions of Dirichlet, Neumann or mixed type,
and to deal with other classes of controls (e.g., distributed controls).
The results here obtained can be extended to more recent notions of derivatives,
e.g., the Atangana--Baleanu operators \cite{AB:16,MR4018654}. 
Another line of research consists to carry out numerical experiments 
illustrating the obtained theoretical results. These and other related 
problems need further investigations, and will be addressed elsewhere.


\section*{\; Acknowledgements}

The present work was supported by Hassan II Academy of Science and Technology (Project N 630/2016), 
Morocco, and by FCT and the Center for Research and Development in Mathematics 
and Applications (CIDMA), project UID/MAT/04106/2019, Portugal.
The authors are very grateful to two anonymous referees, 
for several suggestions and invaluable comments.


\bigskip



\emergencystretch=\hsize

\begin{center}
\rule{6 cm}{0.02 cm}
\end{center}

\end{document}